\documentclass[12pt]{article}

\usepackage{amsmath}
\usepackage{amsfonts}
\usepackage{latexsym}
\usepackage{amssymb}
\usepackage{amsthm}
\usepackage[normalem]{ulem}
\usepackage{csquotes}
\usepackage{lineno,xcolor}
\usepackage[normalem]{ulem}
\usepackage{filecontents}
\usepackage{scalefnt}
\usepackage{tikz}
\usepackage{fancyhdr}
\usepackage[utf8]{inputenc}
\usetikzlibrary{arrows}
\usetikzlibrary{positioning}

\usepackage{lineno,xcolor}
\definecolor{amaranth}{rgb}{0.9, 0.17, 0.31}
\definecolor{bluegray}{rgb}{0.4, 0.6, 0.8}

\makeatletter

\newtheorem*{maintheorem*}{Main Theorem}
\newtheorem{theorem}{Theorem}[section]
\newtheorem{proposition}[theorem]{Proposition}

\newtheorem{lemma}[theorem]{Lemma}

\newtheorem*{theorem*}{Theorem}

\newtheorem{remark}[theorem]{Remark}

\newtheorem*{example*}{Example}
\newtheorem*{conjecture*}{Conjecture}
\def\1{\mathbf 1}

\def\j{\mathbf j}

\def\0{\mathbf 0}

\def\cB{\mathcal B}
\def\cC{\mathcal C}

\def\cL{\mathcal L}

\def\cO{\mathcal O}
\def\cP{\mathcal P}
\def\cX{\mathcal X}

\def\cT{\mathcal T}

\def\cS{\mathcal S}
\def\cQ{\mathcal Q}

\def\bR{{\mathbb R}}

\def\<{\langle}
\def\>{\rangle}
\newcommand\comment[1]{}

\newcommand*{\shifttext}[2]{
  \settowidth{\@tempdima}{#2}
  \makebox[\@tempdima]{\hspace*{#1}#2}
}

\newcommand\redsout{\bgroup\markoverwith{\textcolor{amaranth}{\rule[0.5ex]{2pt}{0.4pt}}}\ULon}

\title{Reconstructing a generalized quadrangle from the Penttila-Williford $4-$class association scheme}
\author{Giusy Monzillo\footnote{Giusy Monzillo: giusy.monzillo@unibas.it\hfill\newline\hspace*{1.4em}
Alessandro Siciliano: alessandro.siciliano@unibas.it\hfill\newline Dipartimento di Matematica, Informatica ed Economia -
Universit\`{a} degli Studi della Basilicata - Viale dell'Ateneo Lucano 10 - 85100 Potenza (Italy).} \footnote{The research was supported by the Italian National Group for Algebraic and Geometric Structures and their Applications (GNSAGA-INdAM). }
\and
Alessandro Siciliano${}^{* \dag}$
}
\date{}

\begin{document}


\maketitle

\thispagestyle{fancy}
\fancyhf{}
\renewcommand{\headrulewidth}{0pt}
\lhead{}

\begin{abstract}
Penttila and Williford constructed a $4-$class association scheme from a generalized quadrangle with a doubly subtended subquadrangle. We show that an association scheme with appropriate parameters and satisfying  some assumption about maximal cliques must be the Penttila-Williford scheme.
\end{abstract}

\section{Introduction}
\comment{
Let $\mathfrak X=(X,\{R_i\}_{0\le i\le d})$ be a (symmetric) association scheme with $d$ classes. For $0\le i\le d$, let $A_i$ be the adjacency matrix of the relation $R_i$, and $E_i$ the $i$-th primitive idempotent of the Bose-Mesner algebra of $\mathfrak X$ which projects on the $i$-th maximal common eigenspace of $A_0,\ldots,A_d$. The  matrices $P$ and $Q$ defined by
\[
(A_0\ A_1\ \ldots \ A_d)=(E_0\ E_1\ \ldots \ E_d)P
\]
and 
\[
(E_0\ E_1\ \ldots \ E_d)=|X|^{-1}(A_0\ A_1\ \ldots \ A_d)Q
\]
are the {\em first} and the {\em second eigenmatrix} of $\mathfrak X$, respectively. 

An association scheme is said to be  $P${\em -polynomial}, or {\em metric}, if, after a reordering of the relations, there are polynomials $p_i$ of degree $i$ such that $A_i = p_i(A_1)$; an association scheme is called $Q${\em -polynomial}, or {\em cometric}, if, after a reordering of the eigenspaces, there are polynomials $q_i$ of degree $i$ such that $E_i =q_i(E_1)$, where multiplication is done entrywise. The reader is referred to \cite{bi,bcn} for further information on association schemes.
} 

A (finite) {\em generalized quadrangle} (GQ) is an incidence structure $\cS=(\cP,\cB, {\rm I})$ where  $\cP$ and $\cB$ are disjoint (nonempty) sets of objects called {\em points} and {\em lines}, respectively, and for which $\rm I\subseteq (\cP,\cB)\times (\cB,\cP)$ is a symmetric point-line incidence relation satisfying the following axioms:
\begin{itemize} 
\item[ (i)] Each point is incident with $t+1$ lines, and two distinct points are incident with at most one line.
\item[(ii)]  Each line is incident with $s+1$ points, and two distinct lines are incident with at most one point.
\item[(iii)] If $x$ is a point and $L$ is a line not incident with $x$, then there is a unique pair $(y, M)\in \cP\times \cB$ such that $x\, {\rm I}\, M\, {\rm I}\, y\, {\rm I}\,  L$.
\end{itemize}

The integers $s$ and $t$ are the {\em parameters} of the GQ, and $\cS$ is said to have {\em order} $(s,t)$.  If $\cS$ has order $(s,t)$, then it follows that $|\cP| = (s + 1)(st + 1)$ and $|\cB| = (t + 1)(st + 1)$. \\ For more details on generalized quadrangles, the reader is referred to \cite{pt}.

It is known that the points of a generalized quadrangle under the relation of collinearity form a strongly regular graph, that is a $2-$class association scheme. 
By using the geometry of generalized quadrangles which satisfy prescribed  properties, it is possible to  construct association schemes with more than two classes. \\
Payne in \cite{payne} constructed a $3-$class association scheme starting from a generalized quadrangle with a quasi-regular point. Subsequently, Hobart and Payne in \cite{HobPay} proved that an association scheme having the same parameters and satisfying an assumption about maximal cliques must be the above $3-$class scheme.
In \cite{gl}, Ghinelli and L\"owe defined a $4-$class association scheme on the points of a generalized quadrangle with a regular point, and they characterized the scheme by its parameters.
The techniques used in the above papers are mostly eigenvalue techniques, and \cite{haem} is a general reference for these.

Penttila and Williford  \cite{pw} constructed an infinite   family of  $4-$class association schemes starting from  a generalized quadrangle of order $(r,r^2)$ with a doubly subtended subquadrangle of order $(r,r)$. These schemes are $Q-$bipartite, not $Q-$antipodal, neither $P-$polynomial nor  the dual of a $P-$polynomial scheme. 
In the spirit of \cite{gl,HobPay},  we characterize these schemes by their parameters under certain assumptions.  \\ The principal references on association schemes are  \cite{bi,bcn}. 
 
 \bigskip 
The paper is structured as follows. Section \ref{sec_1} contains background information on generalized 	quadrangles with a doubly subtended subquadrangle as well as  the parameters of the Penttila-Williford $4-$class scheme $(\cX, \{R_i\}_{i=0,...,4})$ arising from this geometry. The data of the scheme suggest that the relation $R_3$ can be viewed as the collinearity between points in the GQ not in the subGQ. In Section \ref{sec_2} properties of maximal $\{0,3\}$-cliques of the scheme are explored. In particular, by  considering {\em triple intersection numbers}, we prove that for any pair in  $R_3$  there exists a unique maximal $\{0,3\}$-clique (of size $r$)  containing it (Lemma \ref{lem_3}). For any such  a clique $C$, the set $T_C$ of all vertices which are $2-$related to $C$ is taken into account. In Section \ref{sec_3}, by assuming two particular  hypotheses, one on the set of all cliques through any vertex and one on the sets of type $T_C$,  we can prove that  any set $T_C$ is disjoint union of maximal $\{0,3\}$-cliques (Proposition \ref{prop_5}), and  finally we are able to reconstruct the GQ  of order $(r,r^2)$ with a doubly subtended sub{GQ} of order $(r,r)$ (Theorem \ref{main}).

We would like to give a few remarks on the reasons why we need to introduce the two hypotheses. In order to prove Proposition \ref{prop_5}, our first attempt was to use  eigenvalue techniques in the footsteps of \cite{HobPay,bh}. Unfortunately, this method did not produce the wanted result. Hence, we approached the problem according to the idea from \cite{is}. Although also this method did not fully work, it provided   the inspiration for the formulation of the first hypothesis. 
The second hypothesis underlies the definition of the lines of the subtended subGQ which  is completely missing from the scheme. The main difference between our problem and those faced in \cite{HobPay,bh,is} is that, there, {\em a few} lines must be reconstructed from the scheme. More technical details are given in the Appendix \ref{appendix}.

\section{Preliminaries}\label{sec_1}
Let $\cS$ be a GQ of order $(r,r^2)$ with a subGQ $\cS'$ of order $(r,r)$, and $x$ a point of $\cS$ not in $\cS'$.  Then, the set of points of $\cS'$ which are collinear with $x$ form an ovoid of $\cS'$. Following Brown \cite{brown},  this ovoid is said to be {\em subtended by the point} $x$, and it will be denoted by $\cO_x$. An ovoid of $\cS'$ is said to be {\em doubly subtended} provided that it is subtended by exactly two points of $\cS\setminus \cS'$. 
The subGQ $\cS'$ is {\em doubly subtended} in $\cS$ provided that every subtended ovoid of $\cS'$ is doubly subtended. In this case,  $x'$ will denote the other point subtending $\cO_x$, and we refer to $x$ and $x'$ as {\em antipodes}. Lemma 2.3 in \cite{brown} states that $\cS'$ is doubly subtended in $\cS$ if and only if $\cS$ has an involutorial automorphism which fixes $\cS'$ pointwise. This automorphism simply interchanges antipodes while leaving the points of $\cS'$ fixed. 

 It is known  \cite[Corollary 2.2]{brown}  that the size of the intersection of two subtended ovoids $\cO_x$ and $\cO_y$ is either  1, $r+1$ or $r^2+1$ depending only on whether $x$ and $y$ subtend different ovoids and $y$ is collinear with either $x$ or $x'$, or $x$ and $y$ subtend different ovoids and $y$ is collinear with neither $x$ nor $x'$, or $x$ and $y$ subtend the same ovoid.

For the convenience of the reader, we recall here the relations of the Penttila-Williford $4-$class scheme $\cX=(X,\{R_i\}_{i=0}^{4})$. 
Let $\cS=(\cP, \cL,\mathrm I)$  be a GQ of order $(r, r^2)$, $r > 2$, with a doubly subtended subGQ $\cS'=(\cP', \cL',\mathrm I')$ of order $(r,r)$. On the set $X$ of points of $\cS\setminus \cS'$, consider the following relations together with the identity relation $R_0$:
\begin{itemize}
\item[$R_1$:]  $(x,y)\in R_1$ if and only if $x$ and $y$ are not collinear in $\cS$ and $|\cO_x \cap \cO_y|=1$.
\item[$R_2$:] $(x,y)\in R_2$ if and only if $x$ and $y$ are not collinear in $\cS$ and $|\cO_x \cap \cO_y|=r+1$. 
\item[$R_3$:]  $(x,y)\in R_3$ if and only if $x$ and $y$ are collinear in $\cS$.
\item[$R_4$:] $(x,y)\in R_4$ if and only if $\cO_x =\cO_y$.
\end{itemize}

Note that $(x,y)\in R_1$ implies that $y$ is collinear with $x'$, and $(x,y)\in R_2$ implies that $y$ is not collinear with $x'$.

Our idea is to reconstruct the quadrangle $\cS$  with the doubly subtended subGQ  $\cS'$ from the parameters of the scheme, which we report below. 

As usual, $n_i$ denotes the valency of $R_i$,  $p^k_{ij}$ are the intersection numbers of the scheme. Precisely,
\begin{equation*}
 n_1=(r-1)(r^2+1), \ \ \ \ \ n_2=(r^2-2r)(r^2+1),  \ \ \ \ \ n_3=(r-1)(r^2+1), \ \ \ \ \ n_4=1,
\end{equation*}
and the intersection numbers $p^k_{ij}$ are collected in the following tables whose entries are indexed by $(i,j)$:

\begin{table}[htbp]
\centering 
\resizebox{.7\textwidth}{!}{
\begin{tabular}{l|ccccc}
$p^1_{i,j}$ & \hspace{.1in}  1 &\hspace{.1in} 2 &\hspace{.1in} 3 &\hspace{.1in} 4\\[.1in]
\hline\\[.03in] 
1 & \hspace{.1in} $r^2 $ &\hspace{.1in} $r^2(r-2)$ &\hspace{.1in}  $r-2 $ &\hspace{.1in}  $0$ \\[.1in]
 2 & \hspace{.1in}  $ r^2(r-2)$ &\hspace{.1in}  $ r^4-4r^3+5r^2-2r$    &\hspace{.1in}   $r^2(r-2) $ &\hspace{.1in}  $0$ \\[.1in]
 3 & \hspace{.1in} $r-2 $ &\hspace{.1in}  $r^2(r-2) $    & \hspace{.1in}  $ r^2$ &\hspace{.1in}  $1$ \\[.1in]
 4 & \hspace{.1in} $0$ &\hspace{.1in}  $0$    & \hspace{.1in}  $1$ &\hspace{.1in}  $0$
\end{tabular}
}
\end{table}

\bigskip

\begin{table}[htbp]
\centering 
\resizebox{.7\textwidth}{!}{
\begin{tabular}{l|cccc}
$p^2_{i,j}$ & \hspace{.1in}  1 &\hspace{.1in} 2 &\hspace{.1in} 3 &\hspace{.1in} 4 \\[.1in]
\hline\\[.03in] 
1 & \hspace{.1in}  $ r(r-1)$ &\hspace{.1in} $ (r-1)^3$   &\hspace{.1in} $r(r-1)$  &\hspace{.1in} $0$ \\[.1in]
 2 & \hspace{.1in} $(r-1)^3$  &\hspace{.1in} $r^4-4r^3+7r^2-8r$     &\hspace{.1in}  $ (r-1)^3$ &\hspace{.1in} $1$   \\[.1in]
 3 & \hspace{.1in} $r(r-1)$   &\hspace{.1in} $ (r-1)^3$      & \hspace{.1in}  $r(r-1)$  &\hspace{.1in} $0$ \\[.1in]
 4 & \hspace{.1in} $0$   &\hspace{.1in} $1$      & \hspace{.1in}  $0$  &\hspace{.1in} $0$ 
\end{tabular}
}
\end{table}

\bigskip

\begin{table}[htbp]
\centering 
\resizebox{.7\textwidth}{!}{
\begin{tabular}{l|cccc}
$p^3_{i,j}$ & \hspace{.1in}  1 &\hspace{.1in} 2 &\hspace{.1in} 3 &\hspace{.1in} 4 \\[.1in]
\hline\\[.03in] 
1 & \hspace{.1in}  $ r-2$ &\hspace{.1in} $ r^2(r-2)$   &\hspace{.1in} $r^2 $  &\hspace{.1in} $ 1$ \\[.1in]
 2 & \hspace{.1in} $r^2(r-2) $  &\hspace{.1in} $r^4-4r^3+5r^2-2r $     &\hspace{.1in}  $r^2(r-2) $ &\hspace{.1in} $0$   \\[.1in]
 3 & \hspace{.1in} $r^2 $   &\hspace{.1in} $r^2(r-2) $      & \hspace{.1in}  $r-2 $  &\hspace{.1in} $0$ \\[.1in]
 4 & \hspace{.1in} $1$   &\hspace{.1in} $0$      & \hspace{.1in}  $ 0$  &\hspace{.1in} $0 $ 
\end{tabular}

}
\end{table}

\bigskip

\begin{table}[htbp]
\centering 
\resizebox{.75\textwidth}{!}{
\begin{tabular}{l|cccc}
$p^4_{i,j}$ & \hspace{.1in}  1 &\hspace{.1in} 2 &\hspace{.1in} 3 &\hspace{.1in} 4 \\[.1in]
\hline\\[.03in] 
1 & \hspace{.1in}  $0$ &\hspace{.1in} $0$   &\hspace{.1in} $(r-1)(r^2+1)$  &\hspace{.1in} $0$ \\[.1in]
 2 & \hspace{.1in} $0$  &\hspace{.1in} $r(r-1)(r^2+1)$     &\hspace{.1in}  $0$ &\hspace{.1in} $0$   \\[.1in]
 3 & \hspace{.1in} $(r-1)(r^2+1)$   &\hspace{.1in} $0$      & \hspace{.1in}  $0$  &\hspace{.1in} $0$ \\[.1in]
 4 & \hspace{.1in} $0$   &\hspace{.1in} $0$      & \hspace{.1in}  $0$  &\hspace{.1in} $0$ 
\end{tabular}
}
\end{table}

\newpage
The first and second eigenmatrices of the scheme are:
\[
P=\begin{pmatrix}
1 & (r-1)(r^2+1)	 & r(r-2)(r^2+1) & (r-1)(r^2+1) & 1  \\[.05in]
1 & r^2+1	 & 0 & -(r^2+1) & -1  \\[.05in]
1 & r-1 	 & -2r & r-1 & 1 \\[.05in]
1 & -r+1	 & 0 & r-1 & -1  \\[.05in]
1 & -(r-1)^2	 & 2r(r-2) & -(r-1)^2 & 1
 \end{pmatrix}
 \]

\bigskip
\[
Q=\begin{pmatrix}
1 & \frac{r(r-1)^2}{2}	 & \frac{(r-2)(r+1)(r^2+1)}{2} & \frac{r(r-1)(r^2+1)}{2} & \frac{r(r^2+1)}{2}  \\[.05in]
1 & \frac{r(r-1)}{2}	 & \frac{(r-2)(r+1)}{2} & \frac{-r(r-1)}{2} & \frac{-r(r-1)}{2}  \\[.05in]
1 & 0	 & -(r+1) & 0 & r \\[.05in]
1 & \frac{-r(r-1)}{2}	 & \frac{(r-2)(r+1)}{2} & \frac{r(r-1)}{2} & \frac{-r(r-1)}{2}  \\[.05in]
1 & \frac{-r(r-1)^2}{2}	 & \frac{(r-2)(r+1)(r^2+1)}{2} & \frac{-r(r-1)(r^2+1)}{2} & \frac{r(r^2+1)}{2}  
 \end{pmatrix}.
\]

\section{Some properties of maximal $\{0,3\}$-cliques of $\cX$}\label{sec_2}
 
A  $\{0,3\}${\em-clique}  of the association scheme $\cX$ is a subset $Y$ of the vertex set $X$  such that $(x,y)\in R_0\cup R_3$, for all $x,y\in Y$;  a $\{0,3\}$-clique is said to be {\em maximal} if it is not contained in a larger $\{0,3\}$-clique.

Let $xyu$ be a (ordered) triple of elements in  $X$, and  $l,m,n \in \{0,\ldots,d\}$. The {\em triple intersection number} $\left[ \begin{smallmatrix} x & y & u \\l & m & n \end{smallmatrix} \right]$ (or $[\, l\; m\; n \,]$ for short) denotes the number of vertices $z\in X$ such that 
\[
(x,z)\in R_l, \ \ (y,z)\in R_m, \ \ (u,z)\in R_n.
\] 

Note that this symbol is invariant under permutations of its columns. Let $
(x,y)\in R_A$,  $(y,u)\in R_B$, $(u,x)\in R_C$, for some $A,B,C\in\{0,\ldots,d\}$.
The following identities hold:
\begin{equation}\label{eq_3}
\begin{array}{ccccccccc}
[\, 0\; m\; n \,] &+ &[\, 1\; m\; n \,]&+&\ldots &+&[\, d\; m\; n \,]& =&p^B_{mn}, \\[.02in]
[\, l\; 0\; n \,] &+ &[\, l\; 1\; n \,]&+&\ldots&+&[\, l\; d\; n \,]& =&p^C_{ln}, \\[.02in]
[\, l\; m\; 0 \,] &+& [\, l\; m\; 1 \,]&+&\ldots&+&[\, l\; m\; d \,]& =&p^A_{lm},
\end{array}
\end{equation}
for all $l,m,n\in\{0,\ldots,d\}$. Since
\begin{equation}\label{eq_2}
 [\, 0\; m\; n \,] = \delta_{m\,A}\delta_{n\, C},\ \ \  
 [\, l\; 0\; n \,] = \delta_{l\,A}\delta_{n\, B},\ \ \ 
 [\, l\; m\; 0 \,] = \delta_{l\,C}\delta_{m\, B}, 
\end{equation}
the identities (\ref{eq_3}) reduce to
\begin{equation}\label{eq_4}
\begin{aligned}
\sum_{r=1}^{d}{[\, r\; m\; n \,]}& =p^B_{mn}-\delta_{mA}\delta_{n C}, \\[.02in]
\sum_{r=1}^{d}{[\, l\; r\; n \,]}& =p^C_{ln}-\delta_{lA}\delta_{n B},\\[.02in]
\sum_{r=1}^{d}{[\, l\; m\; r \,]}& =p^A_{lm}-\delta_{lC}\delta_{m B},
\end{aligned}
\end{equation}
for all $l,m,n\in\{1,\ldots,d\}$. 
Identities  (\ref{eq_4}) can be interpreted as a system of  equations in the $(d)^3$ non-negative unknowns $[\, l\; m\; n\,]$, and we refer to it as {\em the system of equations associated to the triplet} $ABC$.  We point out the reader \cite{cj} for more details on triple intersection numbers in an association scheme.

Note that the system is uniquely determined by its  constant terms array, which will be denoted by $\left(p^B_{mn}-\delta_{mA}\delta_{n C};p^C_{ln}-\delta_{lA}\delta_{n B};p^A_{ln}-\delta_{lC}\delta_{m B}\right)_{l,m,n\in\{0,\ldots,d\}}$. Instead of the (ordered) triple $xyu$, consider the triple $yxu$. Then, the system of linear equations associated to $A'B'C'=ACB$ is defined by the following constant terms array
\begin{equation}\label{eq_5}
\left(p^C_{mn}-\delta_{mA}\delta_{n B};p^B_{ln}-\delta_{lA}\delta_{n C};p^A_{ln}-\delta_{lB}\delta_{m C}\right).
\end{equation}  
The unknown $[\, i\; j\; k \,]$ in (\ref{eq_4}),  that represents  $\left[ \begin{smallmatrix} x & y & u \\i & j & k \end{smallmatrix} \right]$, corresponds to the unknown $[\, j\; i\; k \,]'$ in the system defined by (\ref{eq_5}), representing  $\left[ \begin{smallmatrix} y & x & u \\j & i & k \end{smallmatrix} \right]$.  Note that if $B=C$, the two systems coincide, so that $[\, i\; j\; k \,]=[\, j\; i\; k \,]'=[\, j\; i\; k \,]$, for all $i,j\in\{0,\ldots,d\}$. This yields more useful conditions on unknowns in (\ref{eq_4}).
\begin{lemma}\label{lem_2}
For $x,y,u\in X$ with $(x,y),(y,u),(u,x)\in R_3$, $[\, i\; 3\; 3\,] =0$,  for  $i=1,2$.
\end{lemma}
\begin{proof}
 As $\cX$ is a $Q-$polynomial $Q-$bipartite scheme with $q^k_{ij}=0$, for $( i\, j\, k)= (1\,1 \, 1) , (1\,3 \, 1) , (1\,4 \, 1), (1\,2 \, 2), (1\,4 \, 2), (1\,3\, 3), (1\,4\, 4)$   and their permutations, by \cite[Theorem 3]{cj} we have
\begin{equation}\label{eq_6}
\sum_{l,m,n=1}^{4}{Q_{lr}Q_{ms}Q_{nt}[\, l\; m\; n\,]}=-Q_{0r}Q_{As}Q_{Ct}-Q_{Ar}Q_{0s}Q_{Bt}-Q_{Cr}Q_{Bs}Q_{0t},
\end{equation}
for $(r\, s\, t)= (1\,1 \, 1) , (1\,3 \, 1) , (1\,4 \, 1), (1\,2 \, 2), (1\,4 \, 2), (1\,3\, 3), (1\,4\, 4)$   and their permutations. The identities (\ref{eq_6}) can be interpreted as a system of equations in the $4^3=64$ non-negative unknowns $[\, l\; m\; n\,]$. 

For $A=B=C=3$,  we widen the system (\ref{eq_4}) with  the identities $[\,l\;m\;n\,]=[\,\sigma(l)\;\sigma(m)\;\sigma(n)\,]$, for any permutation $\sigma$ on symbols $l,m,n$, and  equations (\ref{eq_6}). 
Handing the above equations  to the computer algebra system Mathematica \cite{math}, we obtain their space of solutions,  depending just on $[\, 1\; 3\; 1\,]=0$. This implies  $[\, 1\; 3\; 3\,]=[\, 1\; 3\; 1\,]=0$ and $[\, 2\; 3\; 3\,]=-2[\, 1\; 3\; 1\,]=0$.
\comment{2. For $A=3$, $B=C=1$,  we widen the system (\ref{eq_4}) with  the identities $[\,l\;m\;n\,]=[\,m\;l\;n\,]$, and  Equations (\ref{eq_6}). 
As before,  thanks to  Mathematica \cite{math}, we obtain their space of solutions depending just on $[\, 1\; 3\; 1\,]=0$, for $r\ge3$. This implies also $[\, 1\; 1\; 1\,]=[\, 1\; 3\; 3\,]=[\, 1\; 3\; 1\,]=0$,\ $[\, 1\; 1\; 2\,]=[\, 1\; 2\; 3\,]=-2 [\, 1\; 3\; 1\,]=0$.} \end{proof}
\begin{lemma}\label{lem_3}
Let $x,y\in X$ with $(x, y)\in R_3$. Then, there exists a unique maximal $\{0,3\}$-clique (of size $r$) in $\cX$ containing $x$ and $y$, and this is the set $\{x,y\}\cup \cP^{(x,y)}_{3,3}$, with $\cP^{(x,y)}_{3,3}=\{z\in X:(x,z)\in R_3, (y,z)\in R_3\}$.
\end{lemma}
\begin{proof}
 Clearly, $|\cP^{(x,y)}_{3,3}|=p^{3}_{3,3}= r-2$.  Assume there exists a pair of distinct elements  $z,z'\in \cP^{(x,y)}_{3,3}$, with $(z,z')\in R_i$, $i\in\{1,2,4\}$. Since $p^4_{33}=0$, we have $i\in\{1,2\}$. Then, $xyzz'$ is a $4-$tuple such that $x,y,z$ are pairwise $3-$related and 
\[
(z',x)\in R_3, \ \ (z',y)\in R_3, \ \ (z',z)\in R_i.
\] 
This yields that the number $[\, i\; 3\; 3\,]$ should be nonzero, but this contradicts  Lemma \ref{lem_2}. Then, $\{x,y\}\cup\cP^{(x,y)}_{3,3}$ is the unique maximal $\{0, 3\}$-clique of size $r$ containing $x$ and $y$.
\end{proof}

From now on,  ``clique'' will stand for  ``maximal $\{0, 3\}$-clique''.

For any $x \in X$, we denote by $x'$ the unique (as $p^0_{44}=1$) element which is $4-$related to $x$. We call $x'$ the \emph{antipode} of $x$ and $\{x,x'\}$ an \emph{antipodal pair}. 

Since $p^4_{13}=n_1=n_3$, $(x,y)\in R_1$ if and only if $(x,y')\in R_3$; as $p^4_{22}=n_2$, $(x,y)\in R_2$ if and only if $(x,y')\in R_2$. 

\begin{lemma}\label{lem_4}
For any clique $C$ in $\cX$, $C'=\{x' : x \in C\}$ is a clique of $\cX$.
\end{lemma}
\begin{proof}
The result follows from the chain $(x,y)\in R_3 $ if and only if $(x,y')\in R_1$ if and only if $(x',y')\in R_3$.
\end{proof}
We will refer to the cliques $C$ and $C'$ in Lemma \ref{lem_4} as {\em antipodal cliques}.
Let $Y$ be a subset of  $X$.  We say that  $z\in X$ is $i-$related to $Y$, if  there is $x\in Y$ such that $(x,z)\in R_i$. 
\begin{lemma}\label{lem_7}
Let $C$ be a clique and $z\notin C$  be $3-$related to $C$.  Then, $z$ is $3-$related to precisely one point in $C$ and one point in $C'$. 
 In addition, the set of points $z\notin  C$  which are $3-$related to $C$ has size $r^3(r-1)$.
\end{lemma}
\begin{proof}
The uniqueness of the point in $C$ that is $3-$related to $z$ follows from Lemma \ref{lem_3}.
Let $u\in C$ such that $(z,u)\in R_3$. Because of the previous results, $u$ lies on $r^2$ cliques, $C$ included, different from the one containing $\{u,z\}$. To prove that there exists exactly one point in $C'$ that is $3-$related to $z$ means to prove that there exists exactly one point in $C$ that is $1-$related to $z$. Suppose there exist two distinct points in $C$ that are $1-$related to $z$, then there would be two distinct points in $C'$ that are $3-$related to $z$, and this cannot happen. So there is at most one point in $C$ that is $1-$related to $z$. As $p^3_{31}=r^2$, there exists exactly one point $v\in C$  such that $(z,v)\in R_1$. 

Now we count the points $z\notin C$  which are $3-$related to $C$.
Any point in $C$ is $3-$related to $n_3-(r-1)=r^2(r-1)$ points not in $C$. As a $z\notin  C$ that is $3-$related to $C$  is $3-$related to exactly one point on $C$, the set of all these points $z$ has size $r( n_3-(r-1))=r^3(r-1)$.
\end{proof}
\begin{remark}\label{rem_1}
From the proof of the previous result we note that if $z\notin C$ is $3-$related to $C$, then there are precisely two points $x,y\in C$ such that $(x,z)\in R_3$ and $(y,z)\in R_1$. 
\end{remark} 
\begin{proposition}\label{prop_2}
Let $C$ be a clique in $\cX$ and $\Delta_C$ the set
\[
\Delta_C=\{z: (z,x)\in R_2, \mathrm{\ for\  all\ } x\in C \}.
\]
Then, $T_C=  \Delta_C\cup C \cup C'$ is a set of $r^3-r^2$ points. Furthermore, for every $z\in\Delta_C$, $|R_1(z) \cap\Delta_C|=|R_3(z) \cap\Delta_C|=r-1$ and $z' \in \Delta_C$. 
\end{proposition}
\begin{proof}
By Lemma \ref{lem_7}, $\Delta_C$ has size $|X|- r^3(r-1)-2 r$, thus  $T_ C$ consists of $r^3-r^2$ points.

Fix $z\in \Delta_C$. Since for any given $y\notin T_C$ there is exactly one point in $C$ which is $3-$related to $y$, $C$ provides a partition of the points not in $T_C$ which are $3-$related to $z$ in $r$ sets of size $p^2_{33}=r^2-r$. Therefore, the points of $\Delta_C$ which are $3-$related to $z$ are $n_3-rp^2_{33}=r-1$. Similar arguments show that the points of $\Delta_C$ which are $1-$related to $z$ are $n_1-rp^2_{11}=r-1$. As $(z,x)\in R_2$ if and only if $(z',x)\in R_2$, it follows $z'\in \Delta_C$.
\end{proof}

\section{Reconstructing the generalized quadrangle from the scheme}\label{sec_3}

The aim is to prove that the set $(T_C,R_3)$ is the graph $(r^2-r)K_r$.

\begin{lemma}\cite{is}\label{lem_10}
Let $x\in X$, and $\cQ(x)$ denote the set of the cliques through $x$.  For $y\in R_2(x)$, let $\lambda(y)$ be the set of cliques $D\in \cQ(x)$ such that $y$ is $3-$related to $D$, and $\mu(y)$ be the set of cliques  $D\in \cQ(x)$ such that $D\subset R_2(y)$. Then:
\begin{itemize} 
\item[-] $|\lambda(y)|=p^2_{33}=r(r-1)$,  $|\mu(y)|=r^2+1-p^2_{33}=r+1$;
\item[-] for  $u,v\in R_2(x)$,
\[
|\mu(u)\cap\mu(v)|=m-r^2+2r+1,\\[.02in]
\]
where  $m=m(x;u,v)=|\lambda(u)\cap\lambda(v)|$;
\item[-]
for $(u,v)\in R_3$, 
\[
n=n(x;u,v)=|R_3(x)\cap R_3(u)\cap R_3(v)|\le 1.
\]
\end{itemize}
\end{lemma}
\begin{proof}
By Lemma \ref{lem_7}, $R_{3}(x)\cap R_{3}(y)$ contains no pairs of  points which are $3-$related. Therefore, $|\lambda(y)|=p^2_{33}=r(r-1)$ and $|\mu(y)|=r^2+1-p^2_{33}=r+1$.

Let $u,v\in R_2(x)$  and let $m=|\lambda(u)\cap\lambda(v)|$. By the parameters of $\cX$  there are $p^2_{33}-m=r^2-r-m$ cliques in $\cQ(x)$ which are $3-$related to exactly one point in $\{u,v\}$.  Therefore, 
\begin{equation}\label{eq_7}
|\mu(u)\cap\mu(v)|=r^2+1-2(r^2-r-m)+m=m-r^2+2r+1.
\end{equation}

Let  $y\in R_3(x)$ be a point that is $3-$related to both $u$ and $v$. If $(u,v)\in R_3$, then $y$ belongs to the unique clique defined by $u$ and $v$. Since in the graph induced by $R_3$ on $X$ the set $R_3(x)$ is a disjoint union of cliques of size $r$, such a point $y$ is  unique. In  other words, $n=|R_3(x)\cap R_3(u)\cap R_3(v)|\le 1$.
\end {proof}

\begin{center}
\framebox[1.1\width][c]{\textbf{Hypothesis 1.} For $u, v \in R_2(x)$ with $(u, v)\in R_3$, $m+n=r^2-2r$.}
\end{center}

\bigskip 
Henceforth, we assume Hypothesis 1 in all subsequent results.

\begin{lemma} \label{lem_9}
Let $\cQ$ be a set of size $r^{2}+1$ covered by a family $\{U_i| i=1,\ldots, r\}$ of $r$ subsets of $\cQ$, each of size $r+1$, and pairwise intersecting exactly in one element. Then, all the sets $U_i$ intersect in the same unique element.
\end{lemma}

\begin{proof}
Let $I=\{1,...,r\}$ and $d(x)=|\{i \in I :x\in U_i\}|$. Let $\overline{d}= \underset{x \in \cQ} \max\ d(x)$. By the definition of $U_i$, we have $1<\overline{d} \leq r$. It suffices to show that $\overline{d} = r$ in order to achieve the statement. Suppose $\overline{d}<r$. Let $\overline x \in \cQ$ such that $d(\overline x)=\overline d$. Without loss of generality, we may assume $\overline x \in U_i$, for $i=1,...,\overline d$. We have $|\cup^{\overline d}_{i=1}{U_i}|=r \overline d+1$. Since $|U_k \cap U_l|=1$, for $k \neq l$, every subset $U_j$, with $j>\overline d$, shares precisely $\overline d$ elements with $\cup^{\overline d}_{i=1}{U_i}$. This yields 
\[
r^2+1 -( r \overline d+1) \le \sum^r_{j = \overline d +1}{|U_j \setminus \cup^{\overline d}_{i=1}{U_i} |}=  (r+1-\overline d)(r-\overline d). 
\]
It follows that $(r-\overline d)(1-\overline d) \geq 0$, i.e. $\overline d \leq 1$; but this is impossible as $\overline d>1$.
\end{proof}

\begin{proposition}\label{prop_5}
For any clique $C$, the set $\Delta_C$ is disjoint union of $r^2-r-2$ cliques and it contains no other pair in $R_3$.
\end{proposition}
\begin{proof}
For $x \in \Delta_C$ we shall prove that in $\cQ(x)$ there is exactly one clique which lies in $\Delta_C$ and that  any other clique in $\cQ(x)$ intersects $\Delta_C$ exactly in $x$. 
If a clique $D\in \cQ(x)$ belongs to all the sets $\lambda(v),$ $v\in C$, then each point of
$C$ is $3-$related to some point in $D \setminus{\{x\}}$. As $|C|>|D \setminus \{x\}|$,  at least
one point of $D$ is $3-$related to two points of $C$. This is impossible because of Lemma \ref{lem_7}. So the sets $\mu(v)$, $v\in C$, cover $\cQ(x)$. On the other hand, for  distinct points $u,v\in C$, we have $n=|R_3(x)\cap R_3(u)\cap R_3(v)|=0$ as $x \in \Delta_C$.  By Hypothesis 1, $m=r^2-2r$ which yields $|\mu(u)\cap\mu(v)|=1$ from (\ref{eq_7}). So we have a set $\cQ(x)$ of size $r^{2}+1$ covered by a family
$\{\mu(v)|v\in C\}$ of $r$ subsets, each of size $r+1$, such that any two of them intersect in exactly one element of $\cQ(x)$. Thus, by Lemma \ref{lem_9}, all the sets $\mu(v)$, $v\in C$, have a clique $D$ in common, and so $D$ is entirely contained in $\Delta_C$. 
Finally,  the second part of the statement follows from the {regularity} of $\Delta_C$ with respect to $R_3$ (see Proposition \ref{prop_2}).
\end{proof}

A clique $D$ is said to be \emph{congruent} to a clique $C$ if either $D=C$ or $(D \times C)\cap R_3= \emptyset$. By Proposition  \ref{prop_5} the congruency is an equivalence relation. Let $\cT$ be the set of all equivalence classes. This clearly means that every $T \in \cT$ is defined by any clique $C$ contained in it, so $T=T_C$. 
The techniques used in the following results are mostly eigenvalue techniques, and  \cite{haem} is a general reference for these.
\begin{proposition}\label{prop_3}
Let $T \in \cT$ and  $x\notin T$. Then, $x$ is $3-$related to $r^2-r$ points of $T$.
\end{proposition}
\begin{proof}
Order $X$ so that the matrix $A_1$ is partitioned as follows:
\[
A_1=\begin{pmatrix}
A_{1\,T} & M \\
M^t & A_{1\,X\setminus T}
\end{pmatrix},
\]
where $A_{1\,T}$ and $A_{1\,X\setminus T}$ are the adjacency matrices of $R_1$ restricted to $T$ and $X\setminus T$, respectively. Let $B_1$ be the matrix of average row sums for this partition. Then,
\[
B_1=\begin{pmatrix}
r-1 & r^3-r^2 \\
r^2-r& (r-1)(r^2-r+1)
\end{pmatrix},
\]
which has eigenvalues $(r-1)(r^2+1)$ and $-(r-1)^2$. Since these tightly interlace  the eigenvalues of  $A_1$,  the row sums are constant. It follows that for every $x\notin T$ there are $r^2-r$ points of $T$ that are $1-$related to $x$. 

Similarly write
\[
A_2=\begin{pmatrix}
A_{2\,T} & M \\
M^t & A_{2\,X\setminus T}
\end{pmatrix}.
\]
By using the regularity of $T$ (with respect to $R_1$ and $R_3$),  the matrix  $B_2$ of average row sums for this partition is
\[
B_2=\begin{pmatrix}
r(r^2-r-2)& r^2(r^2-3r+2) \\[.02in]
r(r^2-3r+2)& r(r^3-3r^2+4r-4)
\end{pmatrix},
\]
whose  eigenvalues are $r(r-2)(r^2+1)$ and $2r(r-2)$. Since these tightly interlace the eigenvalues of  $A_2$,  the row sums are constant. This implies that for every $x\notin T$ there are  $r(r-1)(r-2)$ points of $T$ which are $2-$related to $x$. 

In conclusion, the number of points of $T$ that are $3-$related to a fixed $x\in X\setminus T$ are $|T|- (r^2-r)- r(r-1)(r-2)=r^2-r$.
\end{proof}

\begin{lemma}\label{lem_8}
Let $T_1,T_2 \in \cT$, $T_1 \neq T_2$, such that $T_1 \cap T_2 \neq \emptyset$. Then, any clique in $T_1$ intersects exactly one clique in $T_2$, and  $|T_1 \cap T_2|=r^2-r$.
\end{lemma}
\begin{proof}
Let $z\in C_1 \cap C_2 $ with $C_i$ a clique in $T_i$,  $i=1,2$. Suppose $D$ to be a clique in $T_1$, $D \cap T_2= \emptyset$. By the construction of $T \in \cT$, for each $x \in D$ there is a unique $y\in C_2$ such that $(x, y)\in R_3$. As $|D|=|C_2|$, $z$ would be $3-$related to $D$, but this is a contradiction. So $D \cap T_2 \neq \emptyset$, for each clique $D$ in $T_1$. If $|D \cap T_2| > 1$, then $D$ is contained in $T_2$  by Proposition \ref{prop_5}, i.e. $T_1=T_2$ that is impossible. Hence, the result follows.
\end{proof}

\begin{proposition}\label{prop_4}
Let $T \in \cT$ and $x\notin T$.  Then, there are $r+1$ cliques through $x$ disjoint from $T$. For any such a clique $D$, $T_D\cap T=\emptyset$.
\end{proposition}
\begin{proof}
By  Propositions \ref{prop_5} and   \ref{prop_3}, $x$ is $3-$related to $r^2-r$ points of $T$ lying on different cliques in $T$. As $x$ is on $r^2+1$ cliques of $X$ by Lemma \ref{lem_3}, then there are exactly $r^2+1-(r^2-r)=r+1$ cliques through $x$ that are disjoint from $T$. 
Let $D$ be such a clique. Then $D'\cap T=\emptyset$, where $D'$ is the antipodal clique of $D$. Each point of $D$ is $3-$related to $r^2-r$ points of $T$. Note that, for $y,z\in D$, $y\neq z$, the two corresponding subsets of $3-$related points in $T$ are disjoint by Lemma \ref{lem_3}. So, each point of $T$ is $3-$related to exactly one point of $D$. This yields $T_D=\Delta_D\cup D\cup D'$ is disjoint from $T$.
\end{proof}

Let $T_1, T_2 \in \cT$ with $T_1 \cap T_2=\emptyset$. For any  $x\notin T_1\cup T_2$ we define $\theta_i(x)$ to be the number of cliques $C$ around $x$ such that $C\cap (T_1\cup T_2)=i$, for $i=0,1,2$. Clearly, $\theta_0(x)+\theta_1(x)+\theta_2(x)=r^2+1$. By Proposition \ref{prop_3}, also $\theta_1(x)+2\theta_2(x)=2(r^2-r)$ holds. 

Proposition \ref{prop_5} yields that the set $\cC$ of all cliques with one vertex in $T_1$ and one in $T_2$ has size  $r(r^2-r)^2$. By double counting the pairs $(x,C)$ with $x\notin T_1\cup T_2$, $C\in \cC$ with $x\in C$, we get
\begin{equation}\label{eq_8}
\sum_{x\notin T_1\cup T_2}{\theta_2(x)}=r(r-2)(r^2-r)^2.
\end{equation}

By plugging the above equations into Eq. (\ref{eq_8}), we find 
\[
\sum_{x\notin T_1\cup T_2}{\theta_0(x)}=(r^2-r)^2=|X\setminus (T_1\cup T_2)|.
\]

\begin{center}
\framebox[1.1\width][c]{ 
\textbf{Hypothesis 2.} For all $x\notin T_1\cup T_2$, $\theta_0(x)\ge 1$.}
\end{center}

From now on, we assume Hypothesis 2, together with Hypothesis 1,  in all subsequent results.

By Hypothesis 2, for $T_1, T_2 \in \cT$ with $T_1 \cap T_2=\emptyset$, there exists a unique clique that is disjoint from $T_1 \cup T_2$ through any $x \notin T_1\cup T_2$. This means that every such a pair $T_1, T_2$ determines a unique partition of $X$ in elements of $\cT$. This can be viewed in the following way. 

For a fixed $x\notin T_1\cup T_2$, let $C$ be the unique clique on $x$ disjoint from $T_1 \cup T_2$, and $T_3=T_C$. Let $y \notin T_1 \cup T_2 \cup T_3$, and $D$ be the unique clique through $y$ disjoint from $T_1 \cup T_2$. If there was $z \in D \cap T_3$, then there would be two distinct cliques on $z$ disjoint from $T_1 \cup T_2$ since $T_3$ is partitioned in cliques; but this gives a contradiction.

 Let $\Pi$ be the set of such partitions. Note that every partition has size $r+1$.

\begin{theorem} \label{main}
Consider the following incidence structure $\cS$:

\begin{tabular}{lll}
Points: & (i) & elements of $X$ \\[0.1in]
        & (ii)& elements of $\cT$\\[0.1in]
Lines: & (a) & $C \cup \{T_C\}$, where $C$ is a clique in $X$\\[0.1in]
       & (b) & elements of $\Pi$\\[0.1in]
Incidences:& (i),(a) & a point $x$ of type (i) is incident with a line $C\cup \{T_C\}$ \\  && of type (a) if and only if $x\in C$ \\[0.1in]
             &(i),(b)& none \\[0.1in]
             & (ii),(a) &  a point $T$ of type (ii) is incident with a line $C\cup \{T_C\}$ \\ && of type (a) if and only if $T=T_C$\\[0.1in]
						 & (ii),(b) &  a point $T$ of type (ii) is incident with a line $\pi$ of type (b) \\ && if and only if $T\in \pi$.
\end{tabular}

Then, $\cS$ is a GQ of order $(r,r^2)$. Furthermore, the points of type (ii) together with the lines of type (b) give rise to a doubly subtended subGQ $\cS'$ of order $(r,r)$ of $\cS$.
\end{theorem}
\begin{proof}
By Lemma \ref{lem_3}, every point of $X$ lies on $r^2+1$ cliques. Proposition \ref{prop_5} says that any element $T \in \cT$ contains $r^2-r$ cliques, and $T$ is contained in $r+1$ elements of $\Pi$ by Proposition \ref{prop_4}. Hence, every point of $\cS$ is on $r^2+1$ lines. 

It is easy to check that each line contains $r+1$ points.

Let $x \in X$ and consider $C \cup\{ T_C\}$ such that $x \notin C$.  If $x\in T_C$, there exists a unique clique through $x$, say $C_x$, in $T_C$ by Proposition \ref{prop_5}. Thus, 
\[
x \ {\rm I}\ ( C_x \cup \{T_C\})\ {\rm I} \ T_C \ {\rm I} \ ( C \cup \{T_C\}).
\]

If $x\notin T_C$, there exists a unique point $y \in C$ such that $(x, y)\in R_3$. Let $D$ be the clique containing $\{x, y\}$. Then,
\[
x \ {\rm I}\ ( D \cup \{T_D\})\ {\rm I} \ y \ {\rm I} \ ( C \cup \{T_C\}).
\]
Let $x\in X$ and $\pi \in \Pi$. Then, there exists a unique $T \in \pi$ such that $x \in T$. By Proposition \ref{prop_5}, there exists a unique clique $C$ in $T$ containing $x$. Therefore,
\[
x \ {\rm I}\ ( C \cup \{T\})\ {\rm I} \ T \ {\rm I} \ \pi.
\]
Let $T \in \cT$ and  $C \cup \{T_C\}$ with $T_C \neq T$. 

Assume $T \cap T_C \neq \emptyset$. By Lemma \ref{lem_8} there is a unique clique $D$ in $T$ such that $D \cap C=\{z\}$, so $T=T_D$; from which
\[
T \ {\rm I}\ ( D \cup \{T\})\ {\rm I} \ z \ {\rm I} \ ( C \cup \{T_C\}).
\]
Now assume $T \cap T_C = \emptyset$ so that $T, T_C$ define a unique partition $\pi \in \Pi$ of $X$, giving
\[
T\ {\rm I}\ \pi \ {\rm I} \ T_C \ {\rm I} \ ( C \cup \{T_C\}).
\]
Let $T\in \cT$ and $\pi \in \Pi$ with $T \notin \pi$. Since $\pi$ is a partition of $X$,  by Lemma \ref{lem_8} there are precisely $r$ elements in $\pi$ intersecting $T$ in $r^2-r$ points. Thus, there is a unique $T'\in \pi$ disjoint from $T$. Let $\pi'$ denote the partition defined by $T$ and $T'$. Then, the following chain of incidences holds:
\[
T\ {\rm I}\ \pi' \ {\rm I} \ T' \ {\rm I} \ \pi.
\]
Therefore, $\cS$ is a GQ of order $(r,r^2)$, and the points of type (ii) together with the lines of type (b) give rise to a subGQ $\cS'$ of order $(r,r)$ of $\cS$.

Consider the involutorial automorphism $\phi:x\in X \mapsto x' \in X$ of $\cX$. By using Proposition \ref{prop_2}, it is easy to check that $\phi$ induces an involutorial automorphism of $\cS$ which fixes $\cS'$ pointwise. According to \cite[Lemma 2.3]{brown}, $\cS'$ is a doubly subtended subGQ of $\cS$.
\end{proof}
\appendix
\section{Comments on Hypothesis 1 and 2}\label{appendix}

\subsection{Hypothesis 1} 

In order to reconstruct the points of the subtended  subGQ as well as the lines of the GQ through these points, we need to prove that the set $(T_C,R_3)$ is the graph $(r^2-r)K_r$, for each clique $C$. This is equivalent to showing that the matrix $A_3|_{T_C}$ has eigenvalue $r-1$ with multiplicity at least $r^2-r$.

\newpage 
The adjacency matrices $A_i$ have eigenvalues and multiplicities as collected  in the following table
\begin{table}[htbp]
\centering 
\resizebox{.9\textwidth}{!}{
\begin{tabular}{l|ccccc}
 & \hspace{.1in}  $A_1$ &\hspace{.1in} $A_2$ &\hspace{.1in} $A_3$ &\hspace{.1in} $A_4$&\hspace{.1in} multiplicity\\[.1in]
\hline \\[.03in] 
 $V_0$ & \hspace{.1in}$(r-1)(r^2+1)$	 & \hspace{.1in}$r(r-2)(r^2+1)$ & \hspace{.1in}$(r-1)(r^2+1)$ & \hspace{.1in}1 &\hspace{.1in}1 \\[.05in]
$V_1$ & \hspace{.1in}$r^2+1$	 & \hspace{.1in}0 & \hspace{.1in}$-(r^2+1)$ & \hspace{.1in}$-1$  &\hspace{.1in}$\frac{r(r-1)^2}{2}$\\[.05in]
$V_2$ & \hspace{.1in}$r-1$ 	 & \hspace{.1in}$-2r$ & \hspace{.1in}$r-1$ & \hspace{.1in}1&\hspace{.1in} $\frac{(r-2)(r+1)(r^2+1)}{2}$\\[.05in]
$V_3$ & \hspace{.1in}$-r+1$	 & \hspace{.1in}0 & \hspace{.1in}$r-1$ & \hspace{.1in}$-1$ &\hspace{.1in}$\frac{r(r-1)(r^2+1)}{2}$ \\[.05in]
$V_4$ & \hspace{.1in}$-(r-1)^2$	 & \hspace{.1in}$2r(r-2)$ & \hspace{.1in}$-(r-1)^2$ & \hspace{.1in}1 &\hspace{.1in}$\frac{r(r^2+1)}{2}$
\end{tabular}
}
\end{table}

By following Hobart and Payne \cite{HobPay}, the hope is to find a matrix $E$ (involving the matrix $A_3$ and the all-ones matrix $J$, and possibly some other adjacency matrix) with only two eigenvalues so that Theorem 1.3.3 in \cite{hae} can be applied. 
From the table above we see that we do not get only two eigenvalues if we use a linear combination of $A_3$ and $J$.  Anyway, also by taking one further adjacency matrix  the problem persists. The best we can do is to find a matrix $E$ with three distinct eigenvalues, but Theorem 1.3.2 from \cite{hae} does not give useful information. 

The arguments in Section \ref{sec_3} underlying Hypothesis 1 are inspired  by \cite{is}. It follows from the parameters of $\cX$ that $X$ can be considered as a set of vectors $\{x^*=|X|^{1/2}\tilde Ex:x \in X\}$ in the eigenspace $V_1+V_4$ of dimension $n_3+1$; here, $\tilde E=E_1+E_4$ with  $E_i$  the projector over the $i$-th eigenspace, and $x\in X$ is identified with its characteristic vector in $\bR^{|X|}$. If $(x,y)\in R_i$ then $\<x^*,y^*\>=q_1(i)+q_4(i)$. For a fixed $x\in X$, we consider $\{y^*:y \in R_3(x)\}\cup\{x^*\}$.  The hope is that these vectors  form a basis for $V_1+V_4$, and using this fact to evaluate $m+n$ via the inner products $\<v^*,w^*\>$, for $v,w\in R_2(x)$. This product mainly depends on the number of certain configurations whose evaluation involves 
triple intersection numbers which seems to be not constant according to the computations  implemented with Mathematica. In our opinion, this approach appears to be the most fruitful, but more ideas are needed.

\subsection{Hypothesis 2}

The last step toward the characterization of the Penttila-Williford scheme in terms of its parameters is the definition of the lines of the subtended subGQ.  Recall that the points of the subGQ are the sets $T\in\cT$, and through each of them $r+1$ lines of the subGQ must be defined. A hint for the construction of these lines is given by Proposition \ref{prop_4}, a result similar to Lemma 2.7 in \cite{is} and to a Proposition in \cite{HobPay}: for $T\in\cT$ and $x\notin T$ there exists precisely one clique through $x$ which is disjoint from $T$. \\ In this context, it is natural to consider partitions of $X$ into disjoint unions of elements of $\cT$ as the missing lines. In order that two disjoint elements in $\cT$ define a unique line in the subGQ, we need that through a point not in the union of the sets there is precisely one clique that is disjoint from them. From here Hypothesis 2 arose.

\end{document}